\documentclass[12pt]{amsart}

 \newtheorem{thm}{Theorem}[section]
 
 \newtheorem{lem}[thm]{Lemma}
 \newtheorem{prop}[thm]{Proposition}
 \theoremstyle{definition}
 \newtheorem{defn}[thm]{Definition}
 \theoremstyle{remark}
 \newtheorem{rem}[thm]{Remark}
 \theoremstyle{definition}

 \newcommand{\PP}{\mathbb{P}}

\usepackage{xypic} \usepackage{amsthm}

\begin{document}

\title{3-dimensional sundials}
%\today{}

\author[E. Carlini]{Enrico Carlini}
\address[E. Carlini]{Dipartimento di Matematica, Politecnico di Torino, Turin, Italy}
\email{enrico.carlini@polito.it}

\author[M.V.Catalisano]{Maria Virginia Catalisano}
\address[M.V.Catalisano]{DIPTEM - Dipartimento di Ingegneria della Produzione, Termoenergetica e Modelli
Matematici, Universit\`{a} di Genova, Piazzale Kennedy, pad. D
16129 Genoa, Italy.} \email{catalisano@diptem.unige.it}

\author[A.V. Geramita]{Anthony V. Geramita}
\address[A.V. Geramita]{Department of Mathematics and Statistics, Queen's University, Kingston, Ontario, Canada, K7L 3N6 and Dipartimento di Matematica, Universit\`{a} di Genova, Genoa, Italy}
\email{Anthony.Geramita@gmail.com \\ geramita@dima.unige.it  }

%%% ----------------------------------------------------------------------

\begin{abstract}
Robin Hartshorne and Alexander Hirschowitz  proved that a generic collection of lines on $\mathbb P^n$, $n \geq 3$, has bipolynomial Hilbert Function. We extended this result to a specialization of the collection of generic lines, by considering a union of lines and $3$-dimensional sundials (i.e., a union of schemes obtained by degenerating pairs of  skew lines).
\end{abstract}

%%% ----------------------------------------------------------------------
\maketitle
%%% --------------------------------------------------------------------

\section{Introduction} 

In 1982, Robin Hartshorne and Alexander Hirschowitz wrote a beautiful article \cite{HartshorneHirschowitz}  in which they answered the following very natural question:  {\it What is the Hilbert function of a general union of $s$ lines in $\PP ^n$?}

There is a simple natural response, namely that the function is {\it bipolynomial}, i.e. if $X$ is the union of these $s$ generic lines in $\PP^n$ then the Hilbert function of $X$ in degree $d$ is
$$
H(X,d) = \min \{ {d+n\choose n}, s(d+1) \} \ \ \ \hbox{ for every $d$, $s$ and $n$ }
$$
(see \cite{CarCatGer2} for more about bipolynomial Hilbert functions).
The proof that this natural and obvious response is correct is far from obvious.  Indeed, the paper \cite{HartshorneHirschowitz}  consists of a lovely collection of interesting and ingenuous techniques which handle the various cases that arise in the author's method of answering this question.

It is in this paper that some of the foundations were laid for the {\it tour-de-force} by Alexander and Hirschowitz which resulted, in 1995, in their  solution to the classical Problem of Waring for polynomials (see \cite{AH95}).  This break-through result of Alexander and Hirschowitz has, in its turn, led to a long series of results whose goal is the solution of Waring type problems for Segre varieties (see \cite{CGG1},  \cite{CGG2},  \cite{CGG6}, \cite{AOP}), Grassmann varieties (see \cite{CGG3}), Varieties of Reducible Forms (see \cite{CaChGe}) and other kinds of varieties describing much studied classes of tensors.  The importance of all these developments is not solely in their beauty (which is considerable) but in their broad applicability to areas far beyond those purely in algebraic geometry (see \cite{ComMour}, \cite{BuClSh}, \cite{PRW}).  

%inserire Statistic

Thus the original paper of Hartshorne and Hirschowitz is a good example of the proverb ``mighty oaks from little acorns grow", although in this case the ``acorn" is not so little!

It is, on the one hand, because of the fundamental importance of the paper of Hartshorne and Hirschowitz and, on the other hand, because we need a generalization of their result to apply to some new problems, that we needed to revisit their work.  It turns out that a clear but broad generalization of the Hartshorne-Hirschowitz result is needed for a certain approach to Waring type problems for all the Segre and Segre-Veronese varieties.

\section {Preliminary considerations}

In order to explain what we will do in this paper, we will have to explain one of the many techniques used by Hartshorne and Hirschowitz in \cite{HartshorneHirschowitz}.

They first considered the following situation: let $L_1$ and $L_2$ be two general lines in $\PP^n$.  These lines, being general, generate a $\PP^3$ inside $\PP^n$.  Now, pick a point $P \in L_1$ (general) and let $L_2$ move (in the $\PP^3$ generated by it and $L_1$) so as to cross $L_1$ at the point $P$.  One now sees a (degenerate) plane conic -- but that is not the scheme that is the result of this degeneration.  In fact, \cite{HartshorneHirschowitz} show that this movement can be made to take place in a flat family and the result is a scheme that ``remembers" the $\PP^3$ in which we began, in the sense that the limiting scheme is the degenerate conic union the scheme defined by $\wp^2\mid _{\PP^3}$, where $\wp$ is the ideal of the point $P \in \PP^n$ (see 
  \cite[Lemma 2.5]{CarCatGer2}
for a proof for this degeneration).  Thus, we can visualize the limiting scheme as the degenerate plane conic formed by $L_1$ and the limit of $L_2$, along with a direction coming out of the plane of the degenerate conic which (along with the plane) generate the $\PP^3$ which contained $L_1$, $L_2$ at the start.  We have named a scheme formed in this way a {\it sundial}.

Now, one sundial in $\PP^n$ behaves (from the point of view of its Hilbert function) precisely like a pair of disjoint lines in $\PP^n$, but, this is no guarantee that a generic union of $s$ sundials in $\PP^n$ will behave (from the point of view of its Hilbert function) as if it were a generic collection of $2s$ lines in $\PP^n$.

The ``sundial scheme" is clearly {\it less} generic than the analogous scheme of generic lines.  But, despite this, in this paper we will show that, from the point of view of the Hilbert function, $s$ generic sundials behave exactly like 2s generic lines.

This then gives an extension of the Hartshorne Hirschowitz result as well as affording us more freedom in confronting problems involving general unions of linear spaces (both reduced and non-reduced) which we have discussed in \cite{CarCatGer3}.  To be more precise, our main theorem is the following (see Theorem \ref{sundialsinPn})

{\it
Let $X$ be a generic union of $s$ sundials in $\PP^n$.  Then the Hilbert function of $X$ in degree $d$ is bipolynomial and given by 
$$
H(X,d) = \min \left \{ {d+n\choose n}, 2s(d+1) \right \} \ \ \hbox{ for every $s$, $d$ and $n$.}
$$
}

\section{Basic facts and notation}\label{basicsection}

We will always work over an algebraically closed field $k$ of
characteristic zero. Let $R=k[x_0,...,x_n]$ be the coordinate ring
of $\PP^n$, and denote by  $I_X$ the ideal of a closed subscheme $X \subset
\PP^n$. The Hilbert function of $X$ is then  $HF(X,d)=\dim
(R/I_X)_d$ and the Hilbert polynmial  $hp(X,-)$.

\begin{defn} \ Let $X$ be a closed subscheme of $\PP^n$.  We say that $X$ has a {\it bipolynomial Hilbert function} if
\[HF(X,d)=\min\left\{hp(\PP^n,d),hp(X,d)\right\},\]
for all $d\in\mathbb{N}$.
\end{defn}

We often find it more convenient to describe $\dim( I_X)_d$ rather than $HF(X,d)$.

\medskip

Since we will make use of Castelnuovo's inequality several times, we recall it here in a form more suited to our use (for
notation and proof we refer to \cite{AH95}, Section 2).

\begin {defn}  \label{ResiduoTraccia}
If $X, Y$ are closed subschemes of $\mathbb P^n$, we denote by $Res_Y X$
the scheme defined by the ideal $(I_X:I_Y)$ and we call it the
{\it residual scheme} of $X$ with respect to $Y$, we denote by $Tr_Y X \subset Y$ 
the schematic intersection $X\cap Y$, and call it 
the {\it trace} of $X$ on $Y$.
\end {defn}

 \begin{lem} \label{Castelnuovo}{\bf (Castelnuovo's inequality):}
Let $d,\delta \in \mathbb N$, $d \geq \delta$, let ${Y} \subseteq \PP ^n$ be a smooth hypersurface of degree $\delta$,
and let $X \subseteq \PP
^n$ be a  closed subscheme. Then
$$
\dim (I_{X, \PP^n})_d  \leq  \dim (I_{ Res_Y X, \PP^n})_{d-\delta}+
\dim (I_{Tr _{Y} X, Y})_d.
$$
\qed
\end{lem}

The following lemma gives a criterion for adding to a scheme $X
\subseteq \Bbb P^ n$ a set of reduced points lying on a projective variety $V$ and imposing independent
conditions to forms of a given degree in the ideal of $X$ (see also  \cite[Lemma 2.2]{CarCatGer2}).

 \begin{lem} \label{AggiungerePuntiSuY} Let $d \in \Bbb N$ and let $X  \subseteq \Bbb P^n$ be  a closed subscheme.  Let ${Y} \subseteq \Bbb P ^n$ be a closed reduced subscheme, and 
let $P_1,\dots,P_s$ be generic distinct points on $Y$.
If $\dim (I_{X })_d =s$ and $\dim (I_{X +Y})_d =0$, then
$
\dim (I_{X+P_1+\cdots+P_s})_d = 0.$
\par
\qed
\end{lem}

\begin{proof}

By induction on $s$. \\
Since  $(I_{X +Y} )_d= (I_X )_d\cap (I_Y)_d=(0) $ and $\dim (I_{X })_d =s>0$,  let  $f \in (I_{X })_d$,  $f \notin (I_{Y})_d$. Therefore there exists $P \in Y$, $P \notin X$ such that $f(P) \neq 0$. It follows that 
$\dim (I_{X +P})_d = s-1$ and thus the same holds for  a generic point $P_1 \in V$.  So we are done in case $s=1$. 

Let $s>1$ and let $X' = X+P_1$. 
Obviously $\dim (I_{X ' + Y})_d =0$. Hence, by the inductive hypothesis, there exist $s-1$ generic distinct points $P_2, \dots,P_{s}$ in $Y$ such that $\dim (I_{X'+P_2+\cdots+P_{s}})_d =\dim (I_{X+P_1+\cdots+P_s})_d
= 0.$

\end{proof}

\begin {defn}\label{conica degenere}
We say that $C$ is a {\it degenerate conic} if  $C$ is the union
of two intersecting lines $L, M.$  In this case we write
$C=L+M$.
\end {defn}
  
\begin {defn}\label{3dimsundial} 
Let $L$ and $M$ be two intersecting lines in $\mathbb P^n$ ($n \geq 3$), let $P = L \cap M$, and let 
$T\simeq \Bbb P^{3}$ be a generic linear space containing the
scheme $L+M$.  We call the scheme $L+M+ 2P|_T$ a {\it degenerate conic with an embedded point} or a {\it
3-dimensional  sundial}
(see \cite {HartshorneHirschowitz}, or \cite[definition 2.6 with $m=1$]{CarCatGer2} ).

\end {defn}

\medskip

The following lemma shows that a  3-dimensional  sundial in $\mathbb P^n$ is a degeneration of two generic lines   in $\mathbb P^n$ (see  \cite[Lemma 2.5]{CarCatGer2} for the proof in a more general  case).

 \begin{lem}  \label{sundial} 
 Let $X_1  \subset \Bbb P^n $ ($n \geq 3$) be the disconnected  subscheme consisting of two skew  lines $L_1$ and $M$ (so the linear span of $X_1$ is $<X_1> \simeq \Bbb P^{3} $). Then there exists a flat family  of subschemes $$X_{\lambda}\subset <X_1> \ \ \ \ \ (\lambda \in k )$$
whose special fibre $X_0$ is the union of

\begin{itemize}
\item
the line $M $,

\item a line $L$ which intersects  $M$ in a point $P$,

\item   the scheme $2P|_ { <X_1>}$, that is, the schematic intersection of the double point
$2P$ of $\mathbb P^n$ and $<X_1>$.

\end{itemize}
Moreover, if $H \simeq  \Bbb P^{2}$ is the linear span of $L$
and  $M$, then $Res_H(X_0)$ is given by the (simple) point $P$.

 \end{lem}
\qed

\begin{rem} \label{degenerare2rette}
Since  it is easy to see that  in $\mathbb P^n$ ($n \geq 3$) a  3-dimensional  sundial is also  a degeneration of two intersecting lines  and a simple generic point, by the lemma above we get that  in $\mathbb P^n$ ($n \geq 3$) 
a  degenerate conic with an embedded point can be viewed  either as a degeneration of two generic lines, or as a degeneration of a scheme which is the union of a degenerate conic and a simple generic point.

 \end{rem}
Inasmuch as we have upper semicontinuity of the Hilbert function in a flat family, we will use the remark above several times in what follows.

% The first part of the following lemma is an immediate consequence of the fact that: a single multiple point in $\mathbb P^n$ of multiplicity $m$ imposes at most ${m+n-1   \choose n}$ conditions to the forms of degree $d$; a line $d+1$ conditions; a degenerate conic  $2d+1$ conditions;  a degenerate conic with an embedded point $2d+2$ conditions. The second statement of the lemma is obvious.
%A line imposes at most $d+1$ conditions to the forms of degree $d$ in $I_X$; a degenerate conic at most $2d+1$ conditions, and a degenerate conic with an embedded point imposes at most $2d+2$ conditions, then the first part of the following lemma   is clear. The second statement of the lemma is obvious.

Now an obvious, but usefull  observation.

\begin{lem} \label{BastaProvarePers=e,e*}

Let $X= X_1+\dots + X_s \subset \mathbb P^n$ be the  union of non intersecting  closed subschemes $X_i$,  let  $s' <s$ and 
$$X'= X_1+\dots X_{s'} \subset X.$$

\begin{itemize}
\item[(i)]
If  $\dim (I_{X})_d = {d+n \choose n} - \sum_{i=1}^s HF(X_i,d) $ (the expected value), then also $ \dim (I_{X'})_d $ is as expected, that is  
$$\dim (I_{X'})_d = {d+n \choose n} - \sum_{i=1}^{s'} HF(X_i,d) .$$

\item[(ii)]  If  $ \dim (I_{X})_d = 0$, then  $ \dim (I_{X''})_d =0$, for any subscheme $X'' \supset X$ . \par
\end{itemize}
\end{lem} \qed

We now  recall the basic theorem of Hartshorne and Hirschowitz about the Hilbert function of generic lines.
\begin{thm} \cite[Theorem 0.1]{HartshorneHirschowitz} \label{HH}
Let $n, d \in \mathbb N$.
For $n\geq 3$, the ideal of the scheme $X\subset \Bbb P^n$   consisting of $s$ generic
 lines has the expected dimension, that is,
$$
\dim (I_X)_d = \max \left \{ {d+n \choose n} -s(d+1); 0 \right \},
$$
or equivalently
$$
HF(X,d) = \min \left \{ hp(\PP^n,d)={d+n \choose n}, hp(X,d)=s(d+1)
\right \},
$$
that is, $X$ has bipolynomial Hilbert function.
 \end{thm}
\qed
\medskip

To be more precise  the following equivalent  statement is the actual theorem proved in \cite{HartshorneHirschowitz}:

\begin{thm} \cite[Theorem 0.2]{HartshorneHirschowitz} \label{HH2}
Let $n, d \in \mathbb N$. Let
$$ t = \left\lfloor{d+n \choose n} \over {d+1} \right \rfloor ; \ \ \ \ r = {d+n \choose n}-t(d+1)
$$
and let $L_1, \dots, L_{t+1}$ be $t+1$ generic lines in $\Bbb P^n$.
For $n\geq 3$, the ideal of the scheme $X\subset \Bbb P^n$   consisting of the $t$ lines
$L_1, \dots, L_{t}$  and r generic points lying on $L_{t+1}$
 has the expected dimension, that is,
$$
\dim (I_X)_d = 0 .
$$
 \end{thm}
\qed
\medskip

\begin{rem} \label{equivalenzaEnunciati} By Lemma \ref{BastaProvarePers=e,e*}, the statement of Theorem \ref {HH2} easily implies
  the one of Theorem  \ref{HH}; moreover,  by
Lemma \ref{AggiungerePuntiSuY}, it is easy to prove that also the converse holds.

\end{rem}
\medskip

%%%%%%%%%%%%%%%%%%
%%%%%%%%%%%%%%%%%%%

%In  \cite  {CarCatGer2} we laborously  proved the case  $s=d-1$ in $\mathbb P^n$, with  $n\geq 4$.

\section{The main theorem }\label{Risultati}

\begin{prop} \label{sundialsinP3}  Let 
$X\subset \mathbb P^3 $ be the union of  $s$ generic 3-dimensional  sundials and  $l$ generic lines. Then the Hilbert Function of $X$ is 
$$ HF(X,d) = \min \left\{  {d+3 \choose 3}; (d+1)(2s+l)\right \}
,$$
that is, $X$ has bipolynomial Hilbert function.
\end{prop}

\begin{proof}
Let 
 $$t=  \left\lfloor{ {d+3 \choose 3} \over {d+1} } \right\rfloor, 
  \ \ \ \ \   r= {d+3 \choose 3} - t (d+1)
\  \  \mbox{ and } \     s = \left\lfloor{ \frac t 2 } \right\rfloor .$$
By Remark \ref{degenerare2rette}, it sufficies to prove that  the following schemes
have the expected Hilbert Function in degree $d$:
$$W =
   \left \{
  \begin{matrix} 
  \widehat C_1 + \dots +\widehat  C_{  s }  +P_1+\dots +P_r  & {\hbox  {for } } \  t \    {\hbox  {even } }\\
   \widehat C_1 + \dots +\widehat C_{ s }  +M +P_1+\dots P_r  & {\hbox  {for } } \ t \    {\hbox  {odd } }
   \end{matrix}
    \right. ,
     $$
     $$T =
   \left \{
  \begin{matrix} 
  \widehat C_1 + \dots +\widehat  C_{   s }  +M & {\hbox  {for } } \  t \    {\hbox  {even and  }  r>0}\\
   \widehat C_1 + \dots +\widehat C_{ s+1}   & {\hbox  {for } } \ t \    {\hbox  {odd  and } r>0}
   \end{matrix}
    \right. ,
     $$
 where the  $\widehat C_i$ are degenerate conics with an embedded point, that is 3-dimensional sundials, the $P_i$ are generic points and $M$ is a generic line. In other words
$$\dim (I_W)_d =exp\dim (I_W)_d =    {d+3 \choose 3} - t (d+1)-r=0;       
$$
$$\dim (I_T)_d=exp\dim (I_W)_d = \max \left \{   {d+3 \choose 3} - (t +1)(d+1); 0 \right \}=0.$$

Notice that $W$ and $T$, as defined above, each have linear span which is all of $\mathbb P^3$. Thus the lemma is clear for $d=1$.

   We fix the following notation. 
   
    Set  $\widehat C_i= C_i +2R_i$: where $C_i= L_{i,1}+L_{i,2} $ is  the union of the
    two intersecting lines $L_{i,1},L_{i,2}$, and where $2R_i$ is  a double point   with support at $L_{i,1}\cap L_{i,2}$.
\\

  We proceed by induction on $d$.  We omit the easy proof in the cases   $d=2, 3$. 

For $d=4$, we have  $t=7$ and $r=0$, so the scheme $W$ consists of three degenerate conics with embedded points (that is, three $3$-dimensional  sundials) and a line $M$.  
Let $H_i$ be the plane containing the support of $\widehat C_i$. 
By specializing the singular point of $\widehat  C_1$ on $H_2$,  the surfaces of degree $4$ though the $\widehat C_i$ and the line $M$ have the three  $H_i$ as fixed components, and the conclusion easily follows.\\
Now assume $d \geq 5$.
Let $Q$ be a smooth quadric surface.
We consider three cases.
\\
 \par
 {\it Case 1}: $d \equiv 0$ mod 3.  \\
 Let $d = 3h$. We have:
 $$t=   \frac {(h+1)(3h+2)}{ 2},  \ \ \ \ \   r=0 .$$
Note that $\left\lfloor{ \frac t 2 } \right\rfloor\geq 2h+1$, so we can let 
 $\widetilde W$ be the scheme obtained from $W$ by  specializing $2h+1$   sundials  $\widehat C_i$ in such a way that the lines $L_{1,1},\dots ,L_{{2h+1},1}$ become  lines of the same ruling on $Q$, (the lines $L_{1,2},\dots ,L_{{2h+1},2}$ remain generic lines,  not lying on $Q$).  \\
  We have
 $$
 Res_Q  {\widetilde W} =  \left \{
  \begin{matrix} 
L_{1,2}+\dots + L_{{2h+1},2} + \widehat C_{2h+2} + \dots +\widehat  C_{   \frac t 2 } & {\hbox  {for } } \  t \    {\hbox  {even } }\\
L_{1,2}+\dots + L_{{2h+1},2} + \widehat C_{2h+2} + \dots +\widehat  C_{   \frac {t-1} 2 }+M & {\hbox  {for } } \ t \    {\hbox  {odd } }
   \end{matrix}
    \right. .
 $$
 By the inductive hypothesis we have:
 $$\dim (I_{Res_Q  {\widetilde W}} )_{d-2} =  {3h+1 \choose 3}- (3h-1)(2h+1+t-4h-2)=0.
 $$
 
 Now we  consider 
 $Tr_Q  {\widetilde W} $. It  consists of $2h+1$ lines of the same ruling and $p$ generic points,
 where we must determine $p$.

 Since 
 $L_{i,2}$ meets $Q$ in two points (one of which is $L_{i,1} \cap L_{i,2}$), and each $\widehat C_i$
 meets $Q$ in four points ($i \geq 2h+2$), and $M$ meets $Q$ in two points, it is easy to compute that,
 both for $t$ even and $t$ odd, $$p=2(2h+1+t-4h-2) = 3h^2+h.$$
Thinking of $Q$ as $\mathbb P^1 \times \mathbb P^1$, we see that the forms  of degree $d$ in the ideal of 
$Tr_Q  {\widetilde W} $ are  curves of type $(3h-(2h+1), 3h)=(h-1,3h)$ in $\mathbb P^1 \times \mathbb P^1$ passing through
$p$ generic points. Hence
$$\dim ( I_{Tr_Q  {\widetilde W}})_{d} =h(3h+1)-p=0
.$$
So by Lemma \ref {Castelnuovo} and by the semicontinuity of the Hilbert function we get $ \dim (I_{W})_d=   0.$
\\
\par
 {\it Case 2}: $d \equiv 2$ mod 3.  \\
 Let $d = 3h+2$. We have:
 $$t=   \frac {3(h+1)(h+2)}{ 2},  \ \ \ \ \   r=h+1 .$$
 Recall that $ s = \left\lfloor{ \frac t 2 } \right\rfloor$ and so
 $$W =
   \left \{
  \begin{matrix} 
  \widehat C_1 + \dots +\widehat  C_{  s }  +P_1+\dots +P_r  & {\hbox  {for } } \  t \    {\hbox  {even } }\\
   \widehat C_1 + \dots +\widehat C_{ s }  +M +P_1+\dots P_r  & {\hbox  {for } } \ t \    {\hbox  {odd } }
   \end{matrix}
    \right. ,
     $$
     $$T =
   \left \{
  \begin{matrix} 
  \widehat C_1 + \dots +\widehat  C_{   s }  +M & {\hbox  {for } } \  t \    {\hbox  {even and  }  r>0}\\
   \widehat C_1 + \dots +\widehat C_{ s+1}   & {\hbox  {for } } \ t \    {\hbox  {odd  and } r>0}
   \end{matrix}
    \right. ,
     $$
     
Note that   $s\geq 2h+2$ and so we can let
$\widetilde W$ be the scheme obtained from $W$ by  specializing the $r$ points $P_i$ on $Q$ and by  specializing $2h+2$  sundials $\widehat C_i$ in such a way that the
 lines $L_{1,1},\dots ,L_{{2h+2},1}$ become  lines of the same ruling on $Q$. 
 
 The specialization for $T$ proceeds in a slightly different way. First notice that 
  for $h>1$ we have  $s \geq 2h+3$, while for $h=1$ we have $d=5$, so $t=9$ is odd. 
 Thus we can  let $\widetilde T$ be the scheme obtained from $T$ by   specializing $2h+3$  sundials $\widehat C_i$ in such a way that the
 lines $L_{1,1},\dots ,L_{{2h+3},1}$ become  lines of the same ruling on $Q$.
 
  We have
 $$
 Res_Q  {\widetilde W} =  \left \{
  \begin{matrix} 
L_{1,2}+\dots + L_{{2h+2},2} + \widehat C_{2h+3} + \dots +\widehat  C_{   \frac t 2 } & {\hbox  {for } } \  t \    {\hbox  {even } }\\
L_{1,2}+\dots + L_{{2h+2},2} + \widehat C_{2h+3} + \dots +\widehat  C_{   \frac {t-1} 2 }+M & {\hbox  {for } } \ t \    {\hbox  {odd } }
   \end{matrix}
    \right. .
 $$
 
 $$
 Res_Q  {\widetilde T} =  \left \{
  \begin{matrix} 
L_{1,2}+\dots + L_{{2h+3},2} + \widehat C_{2h+4} + \dots +\widehat  C_{ s } +M& {\hbox  {for } } \  t \    {\hbox  {even } }\\
L_{1,2}+\dots + L_{{2h+3},2} + \widehat C_{2h+4} + \dots +\widehat  C_{  s+1 } & {\hbox  {for } } \ t \    {\hbox  {odd } }
   \end{matrix}
    \right. .
 $$

 and by the inductive hypothesis we easily get
 $$\dim (I_{Res_Q  {\widetilde W}} )_{d-2} =\dim (I_{Res_Q  {\widetilde T}} )_{d-2}  =  
 {3h+3 \choose 3}- (3h+1)(t-2h-2)=0.
 $$
The trace 
 $Tr_Q  {\widetilde W} $ consists of $2h+2$ lines of the same ruling and $p$ generic points, 
 where $p$ has to be determined. Since the
 $L_{i,2}$ and $M$  each meet $Q$ in two points  and each $\widehat C_i$
 meets $Q$ in four points, it is easy to compute that 
 $$p=3h^2+5h+2 +r = 3h^2+6h+3.
 $$
Thinking of $Q$ as $\mathbb P^1 \times \mathbb P^1$, the forms  of degree $d$ in the ideal of 
$Tr_Q  {\widetilde W} $ are  curves of type $(3h+2-(2h+2), 3h+2)=(h,3h+2)$ in $\mathbb P^1 \times \mathbb P^1$ passing through
$p$ generic points. Hence
$$\dim ( I_{Tr_Q  {\widetilde W}})_{d} =(h+1)(3h+3)-p=0
.$$

The trace 
 $Tr_Q  {\widetilde T} $ consists of $2h+3$ lines of the same ruling and 
 $$2(t-2h-2)=3h^2+5h+2$$ generic points. 
 Thinking of $Q$ as above, the forms  of degree $d$ in the ideal of 
$Tr_Q  {\widetilde T} $ are  curves of type $(3h+2-(2h+3), 3h+2)=(h-1,3h+2)$ in $\mathbb P^1 \times \mathbb P^1$ passing through
$3h^2+5h+2$ generic points. Hence
$$\dim ( I_{Tr_Q  {\widetilde T}})_{d} =\max \{h(3h+3)-(3h^2+5h+2); 0\}=0
.$$

Hence by Lemma \ref {Castelnuovo} we get $ \dim (I_{\widetilde W})_d=   \dim (I_{\widetilde T})_d=0.$
By the  semicontinuity of the Hilbert function we get the conclusion.
\\

 {\it Case 3}: $d \equiv 1$ mod 3.  \\
 Although in this case we have $r=0$, and so we only  have to deal with $W$,  this is the most difficult case. 
 
 If we write $d = 3h+1$, we have:
 $$t=   \frac {(h+1)(3h+4)}{ 2},  \ \ \ \ \   r=0 .$$

For $h=2$, that is for $d=7$,  
we have
$$
 W= \widehat C_1 + \dots +\widehat C_7  +M .
$$
In this case a direct computation by   \cite{cocoa} gives the conclusion. 

It is also possible to get the conclusion by an ad hoc specialization:
degenerate the line $M$ onto the plane containing the support of $ \widehat C_1  $,  say  $H$, so that the new scheme $ \widehat C_1 +M$ is  the union of three intersecting lines and three double points. Then specialize onto $H$ the singular point of each  of $ \widehat C_2, \widehat C_3$ and $ \widehat C_4$.  Now $H$ is a fixed component for the surfaces of degree $7$ passing through the specialized scheme. Removing that component, we are left with: three generic sundials; three degenerate conics with only their singular point on the plane $H$ and three generic points on $H$. 
By specializing the three generic points  on $H$ to the singular points of the three degenerate conics, we obtain six generic sundials. From Case 1 we are done.

Now assume $h>2$. 
 Note that   for $h>2$ we get  $\left\lfloor{ \frac t 2 } \right\rfloor\geq 4h+1$, so we can let
$\widetilde W$ be the scheme obtained from $W$ by  specializing $2h+1$   sundials $\widehat C_{1}, \dots, \widehat C_{2h+1}$ in such a way that: the
 lines $L_{1,1},\dots ,L_{{2h+1},1}$ become  lines of the same ruling on $Q$;
 the lines $L_{1,2},\dots ,L_{{2h+1},2}$ remain generic lines,  not lying on $Q$;
 the $2h$ points $R_{2h+2}, \dots, R_{4h+1}$ are placed onto $Q$
 (recall that  $R_i$ is the support of the embedded point of $\widehat C_i$).
\\
 
  We have
  \\
  $\bullet$ for $t$ even:
   $$
 Res_Q  {\widetilde W} = 
L_{1,2}+\dots + L_{{2h+1},2} 
+ C_{2h+2} + \dots +  C_{4h+1   } +
\widehat C_{4h+2} + \dots +\widehat  C_{   \frac t 2 } ;
 $$
  $\bullet$ for $t$ odd:
   $$
 Res_Q  {\widetilde W} = 
L_{1,2}+\dots + L_{{2h+1},2} 
+ C_{2h+2} + \dots +  C_{4h+1   } +
\widehat C_{4h+2} + \dots +\widehat  C_{   \frac {t-1} 2 }+M.
  $$
  
 The trace 
 $Tr_Q  {\widetilde W} $ consists of $2h+1$ lines of the same ruling, $2h$ double points  and $p$ generic points, where (as in the previous cases) it is easy to compute that 
 $$p=3h^2-h+2.
 $$
Again thinking of $Q$ as $\mathbb P^1 \times \mathbb P^1$,  the forms  of degree $d$ in the ideal of 
$Tr_Q  {\widetilde W} $ are  curves of type $(h, 3h+1)$ in $\mathbb P^1 \times \mathbb P^1$ passing through
$p$ generic points and $2h$ double points. Hence by {\cite [Theorem 2.1]{CGG5} } we get
$$\dim ( I_{Tr_Q  {\widetilde W}})_{3h+1} =(h+1)(3h+2)-p-6h=0
.$$

Now we have to compute the dimension of $ I_{Res_Q  {\widetilde W}}$ in degree $3h-1$. \\
In order to do that, we specialize $Res_Q  {\widetilde W}$ further into a scheme  ${\widetilde{\widetilde W}}$
in this way:
 specialize the following $2h+1$ lines 
$$L_{1,2},L_{2h+2,1},\dots ,L_{{4h+1},1}$$   so that these lines become lines of the same ruling on $Q$. By recalling  that the singular points of the degenerate conics $C_{2h+2}, \dots , C_{4h+1   } $ lie on $Q$, 
we have
  \\
  $\bullet$ for $t$ even:
   $$
 Res_Q  {\widetilde{\widetilde W}} = 
L_{2,2}+\dots + L_{{2h+1},2} 
+ L_{{2h+2},2}  + \dots + L_{{4h+1},2} +
\widehat C_{4h+2} + \dots +\widehat  C_{   \frac t 2 } ;
 $$
  $\bullet$ for $t$ odd:
   $$
 Res_Q {\widetilde{\widetilde W}}= 
L_{2,2}+\dots + L_{{2h+1},2} 
+  L_{{2h+2},2}  + \dots + L_{{4h+1},2}  +
\widehat C_{4h+2} + \dots +\widehat  C_{   \frac {t-1} 2 }+M.
  $$
 By the inductive hypothesis we have:
 $$\dim (I_{Res_Q  {\widetilde{\widetilde W} } } )_{d-4} =  {3h \choose 3}- (3h-2)(t-4h-2)=0.
 $$
The trace
 $Tr_Q  {\widetilde{\widetilde W} } $ is the union  of $2h+1$ lines of the same ruling and $2(t-5h-2)$ generic points. \\
 As usual, thinking of $Q$ as $\mathbb P^1 \times \mathbb P^1$, we have that the forms  of degree $d-2$ in the ideal of 
$Tr_Q  {\widetilde{\widetilde W}} $ are  curves of type $(3h-1-(2h+1), 3h-1)=(h-2,3h-1)$ in $\mathbb P^1 \times \mathbb P^1$ passing through $2(t-5h-2)$
 generic points. Hence
$$\dim ( I_{Tr_Q  {\widetilde{\widetilde W}}})_{d-2} =(h-1)(3h)-2(t-5h-2)=0
.$$
So by Lemma \ref {Castelnuovo} we get $ \dim (I_{Res_Q  {\widetilde W}})_{d-2}=   0,$
and so by Lemma \ref {Castelnuovo} again we get  $ \dim (I_{\widetilde W})_{d}=   0$, hence $ \dim (I_{W})_{d}=   0$,
and that  finishes the proof .
\\
\par
\end{proof}

%%%%%%%%%%%%%%%%%%%%%%%%%%%%%%%%%
%%%%%%%%%%%%%%%%%%%%%%%%%%%%%%%%%%%

\begin{thm} \label{sundialsinPn}  Let $n \geq 3$ and let
$X\subset \mathbb P^n $ be the union of  $s$ generic 3-dimensional  sundials and  $l$ generic lines. Then the Hilbert Function of $X$ is 
$$ HF(X,d) = \min \left\{  {d+n \choose n}; (d+1)(2s+l)\right \}
,$$
that is, $X$ has bipolynomial Hilbert function.
\end{thm}

\begin{proof}
Let 
 $$t=  \left\lfloor{ {d+n \choose n} \over {d+1} } \right\rfloor, 
  \ \ \ \ \   r= {d+n \choose n} - t (d+1) \ \  \mbox{ and } \ \  
    s = \left\lfloor{ \frac t 2 } \right\rfloor .$$
As in the case $n=3$, by Remark \ref{degenerare2rette}, it sufficies to prove that  the following schemes
have the expected Hilbert Function in degree $d$:
$$W =
   \left \{
  \begin{matrix} 
  \widehat C_1 + \dots +\widehat  C_{  s }  +P_1+\dots +P_r  & {\hbox  {for } } \  t \    {\hbox  {even } }\\
   \widehat C_1 + \dots +\widehat C_{ s }  +M +P_1+\dots P_r  & {\hbox  {for } } \ t \    {\hbox  {odd } }
   \end{matrix}
    \right. ,
     $$
     $$T =
   \left \{
  \begin{matrix} 
  \widehat C_1 + \dots +\widehat  C_{   s }  +M & {\hbox  {for } } \  t \    {\hbox  {even and  }  r>0}\\
   \widehat C_1 + \dots +\widehat C_{ s+1}   & {\hbox  {for } } \ t \    {\hbox  {odd  and } r>0}
   \end{matrix}
    \right. ,
     $$
 where the  $\widehat C_i$ are degenerate conics with an embedded point, that is 3-dimensional sundials, the $P_i$ are generic points and $M$ is a generic line. In other words
$$\dim (I_W)_d =exp\dim (I_W)_d =    {d+n \choose n} - t (d+1)-r=0;       
$$
$$\dim (I_T)_d=exp\dim (I_W)_d = \max \left \{   {d+n \choose n} - (t +1)(d+1); 0 \right \}=0.$$

We will again be using Castelnuovo's  inequality (see Lemma \ref{Castelnuovo}) and specialization, but this time our specializations will all be into
 hyperplanes. For this reason we must also keep track of the following integers:

$$t'=  \left\lfloor{ {d-1+n \choose n} \over {d} } \right\rfloor, 
  \ \ \ \ \   r'= {d-1+n \choose n} - t'd     \ \  \mbox{ and } \ \ 
  s' = \left\lfloor{ \frac {t'} 2 } \right\rfloor .$$
We denote the sundial $\widehat C_i$ by 
   $$\widehat C_i= C_i +2R_i|_{H_i}$$
   where $C_i= L_{i,1}+L_{i,2} $ is  the union of the
    two intersecting lines $L_{i,1},L_{i,2}$;  where   $H_i \simeq \mathbb P^3$ is a   generic linear space containing 
     $L_{i,1}$ and $ L_{i,2}$; and where $2R_i|_{H_i}$ is  a double point  in $H_i$ with support at $L_{i,1}\cap L_{i,2}$.

We proceed by induction on $n+d$.  In Proposition \ref{sundialsinP3} we have proved the case $n=3$ and,
since $W$ and $T$ each have linear span which is all of $\mathbb P^n$,  the theorem is clear for $d=1$. 

So assume $n >3$, $d>1$.

We split the proof into three cases.

\medskip
{\it Case a):} $t$ and $t'$ both odd.\\

In this case we have
$$W =
   \widehat C_1 + \dots +\widehat C_{ s }  +M +P_1+\dots P_r  \ ,     $$
     $$T =
   \widehat C_1 + \dots +\widehat C_{ s+1}      . $$
   
 Notice that $s-s'-r' \geq 0$ (this inequality is treated in the Appendix,  Lemma \ref{ss'rr'}(a)). 
 
 Using this inequality we can construct
$\widetilde W$, a specialization of $W$,
as follows.
Pick a generic hyperplane $H$:
\\
$\bullet$   specialize the first $r'$ sundials,  $\widehat C_1, \dots , \widehat C_{r'}$, in such a way that $L_{i,1}+L_{i,2}\subset H$, but   $2R_i|_{H_i} \not\subset H $, for $1 \leq i \leq r'$; \\
$\bullet$  specialize the next   $s - s' - r'$  sundials, $\widehat C_{r'+1}$,  $\dots, \widehat C_{s-s'}$,  into $H$ ; \\
$\bullet$  specialize the points $P_1, \dots, P_r$  into $H$; \\
$\bullet$ leave the remaining sundials and the line generic.

Similarly, we specialize $T$ to  
 $\widetilde T$ by: \\
 $\bullet$ specializing the first   $r'$ sundials,  $\widehat C_1, \dots , \widehat C_{r'}$, in such a way that $L_{i,1}+L_{i,2}\subset H$, but  $2R_i|_{H_i} \not\subset H $; \\
 $\bullet$  specializing  the next  $s - s' - r'$  sundials,   $\widehat C_{r'+1}$,  $\dots, \widehat C_{s-s'}$, into $H$; \\
  $\bullet$  specializing  the  sundial $\widehat C_{s+1}$  in such a way that $L_{s+1,1} \subset H$, but  $L_{s+1,2} \not\subset H $.

 We will prove that 
  $$\dim (I_{Res_H  {\widetilde W} })_{d-1} = \dim  (I_{Tr_H  {\widetilde W} })_{d}=\dim  (I_{Res_H  {\widetilde T} })_{d-1}=
 \dim   (I_{Tr_H  {\widetilde T} })_{d}=0.$$

  We have:
$$
 Res_H  {\widetilde W} = 
   R_1 + \dots + R_{ r' }   + \widehat C_{s-s'+1} + \dots +\widehat C_{ s } +M ;
$$
$$
 Res_H  {\widetilde T} =  R_1 + \dots + R_{ r' }   + \widehat C_{s-s'+1} + \dots +\widehat C_{ s } +L_{s+1,2};
$$
$$
 Tr_H  {\widetilde W} = 
 C_1 + \dots + C_{ r' }   + \widehat C_{r'+1} + \dots +\widehat C_{ s-s' } +P_1+\dots P_r +  $$$$
 + Tr_H ( \widehat C_{s-s'+1} + \dots +\widehat C_{ s } +M);
$$
$$
 Tr_H  {\widetilde T} =  C_1 + \dots + C_{ r' }   + \widehat C_{r'+1} + \dots +\widehat C_{ s-s' } +L_{s+1,1}  $$
 $$
 + Tr_H ( 2R_{s+1}|_{H_{s+1}} + \widehat C_{s-s'+1} + \dots +\widehat C_{ s } ).
$$
So we have that $Res_H  {\widetilde W} $ and $Res_H  {\widetilde T} $ are the union of 
$s'$ generic sundials, a generic  line and the $r'$ points $ R_1, \dots , R_{ r' } $.  These $r'$ points 
lie on $H$, and are  generic points  on $H \simeq \mathbb P^{n-1}$, but  since $r' \leq d-1$ and $n\geq 4$, for $d \leq 5$ these points are generic also on $\mathbb P^n$. So  for $d \leq 5$ by the inductive hypothesis we immediately get 
$$\dim (I_{Res_H  {\widetilde W} })_{d-1} =\dim (I_{Res_H  {\widetilde T} })_{d-1}=0.
$$
For $d>5$, consider $Res_H  {\widetilde W } -(R_1 + \dots + R_{ r' }) $ and 
$Res_H  {\widetilde T } -(R_1 + \dots + R_{ r' })$.
These schemes are the union of 
$s'$ generic sundials and a generic  line, hence by the inductive hypothesis we have
$$\dim (I_{Res_H  {\widetilde W} -(R_1 + \dots + R_{ r' })})_{d-1} =\dim (I_{Res_H  {\widetilde T} -(R_1 + \dots + R_{ r' })})_{d-1}=r'.
$$
Moreover 
$$\dim (I_{Res_H  {\widetilde W} -(R_1 + \dots + R_{ r' })+H })_{d-1} =
\dim (I_{Res_H  {\widetilde T} -(R_1 + \dots + R_{ r' })+H })_{d-1} 
$$
$$=
\dim (I_{ Res_H  {\widetilde W} -(R_1 + \dots + R_{ r' }) }      )_{d-2}=
\dim (I_{ Res_H  {\widetilde T} -(R_1 + \dots + R_{ r' }) }      )_{d-2}
$$
$$=\max \left \{
{d-2+n \choose n} - (2s'+1)(d-1) ; 0 \right \}=$$
$$=\max \left \{
{d-2+n \choose n} - t'(d-1) ; 0 \right \}= 0,$$
 (the last equality is proved in the Appendix, Lemma \ref{disuguaglianza1} ). 
 
Hence, by
 Lemma \ref{AggiungerePuntiSuY} (with $Y=H$ )  we get 
 $$\dim (I_{Res_H  {\widetilde W} })_{d-1} =\dim (I_{Res_H  {\widetilde T} })_{d-1}=0.
$$
 
 Now we compute $\dim (I_{Tr_H  {\widetilde W} })_{d} $ and $\dim (I_{Tr_H  {\widetilde T} })_{d} $.
 
  Recall that
  $$
 Tr_H  {\widetilde W} = 
 C_1 + \dots + C_{ r' }   + \widehat C_{r'+1} + \dots +\widehat C_{ s-s' } +P_1+\dots P_r +  $$$$
 + Tr_H ( \widehat C_{s-s'+1} + \dots +\widehat C_{ s } +M);
$$
 so the trace $Tr_H  {\widetilde W} $ is the union of  $r'$ degenerate conics, $s-s'-r'$ sundials and 
 $r+2s'+1=r+t'$ generic points. Since $r+t' \geq r'$ (see the Appendix, Lemma \ref{ss'rr'}(b)), by Remark \ref{degenerare2rette}, the dimension, in degree $d$, of $Tr_H  {\widetilde W} $ is not more than the dimension, in   degree $d$, of a scheme which is the  union of $s-s'$ sundials and 
 $r+t'-r'$ generic points. That is, by the inductive hypothesis,
  $$\dim (I_{Tr_H  {\widetilde W} })_{d} \leq {d+n-1 \choose {n-1}}- 2(s-s')(d+1)-(r+t'-r')
  $$
  $$={d+n-1 \choose {n-1}}- (t-t')(d+1)-  {d+n \choose {n}}+t(d+1)-t'
  $$
  $$+{d+n-1 \choose {n}}-t'd =0.
  $$
  It follows that $\dim (I_{Tr_H  {\widetilde W} })_{d} =0$.
  
 Now  recall that
$$
 Tr_H  {\widetilde T} =  C_1 + \dots + C_{ r' }   + \widehat C_{r'+1} + \dots +\widehat C_{ s-s' } +L_{s+1,1}   $$
 $$
 + Tr_H ( 2R_{s+1}|_{H_{s+1}}  +\widehat C_{s-s'+1} + \dots +\widehat C_{ s } ),
$$
that is, the trace $Tr_H  {\widetilde T} $ is the union of  $r'$ degenerate conics, $s-s'-r'$ sundials, a generic line, the scheme $2R_{s+1}|_{H_{s+1} \cap H}$ and 
 $2s'$ generic points. Since $2s' \geq r'$ (see the Appendix, Lemma \ref{ss'rr'}(c)), by Remark \ref{degenerare2rette}, the dimension in degree $d$ of $Tr_H  {\widetilde T} $ is not more than the dimension, in   degree $d$, of the  scheme which is the  union of $s-s'$ sundials, a line  and 
 $2s'-r'$ generic points (we ignore the scheme $2R_{s+1}|_{H_{s+1} \cap H} $), that is, by the inductive hypothesis,
  $$\dim (I_{Tr_H  {\widetilde T} })_{d} \leq {d+n-1 \choose {n-1}}- 2(s-s')(d+1)-(d+1) - (2s'-r')
  $$
   $$= {d+n-1 \choose {n-1}}- (t-t')(d+1)-d - t'+r'
  $$
     $$= {d+n-1 \choose {n-1}}- t(d+1)+t'd-d +{d+n-1 \choose {n}}-t'd
  $$
  $$={d+n \choose {n}}- t(d+1)-d  = r-d \leq 0.$$

  It follows that $\dim (I_{Tr_H  {\widetilde T} })_{d} =0$.

So we have proved that 
  $$\dim (I_{Res_H  {\widetilde W} })_{d-1} = \dim  (I_{Tr_H  {\widetilde W} })_{d}=0,
  $$
  $$\dim  (I_{Res_H  {\widetilde T} })_{d-1}=
 \dim   (I_{Tr_H  {\widetilde T} })_{d}=0,$$
 hence, by Lemma \ref{Castelnuovo}, we are done.
\\

%%%%%%%%%%%%%%%

{\it Case b):} $t'$ even.\\
 Notice that $s-s'-r' \geq 0$ (see  the Appendix,  Lemma \ref{ss'rr'}(a)). 
 Using this inequality we construct
$\widetilde W$, a specialization of $W$,
as follows.
Let $H$ be  a generic hyperplane:
\\
$\bullet$   specialize the  $r'$ sundials  $\widehat C_1, \dots , \widehat C_{r'}$, in such a way that $L_{i,1}+L_{i,2}\subset H$, but   $2R_i|_{H_i} \not\subset H $, for $1 \leq i \leq r'$; \\
$\bullet$  specialize the   $s - s' - r'$  sundials $\widehat C_{r'+1}$,  $\dots, \widehat C_{s-s'}$  into $H$ ; \\
$\bullet$  specialize the points $P_1, \dots, P_r$  into $H$; \\
$\bullet$   if $t$ is odd, specialize the line $M$ into $H$.\\

Similarly, we specialize $T$ to  
 $\widetilde T$ by: \\
 $\bullet$ specializing the first   $r'$ sundials,  $\widehat C_1, \dots , \widehat C_{r'}$, in such a way that $L_{i,1}+L_{i,2}\subset H$, but  $2R_i|_{H_i} \not\subset H $; \\
 $\bullet$   if $t$ is odd, specializing  the $s+1 - s' - r'$  sundials   $\widehat C_{r'+1}$,  $\dots, \widehat C_{s+1-s'}$ into $H$; \\
   $\bullet$ if $t$ is even, specializing  the $s - s' - r'$  sundials   $\widehat C_{r'+1}$,  $\dots, \widehat C_{s-s'}$ and the line $M$ into $H$.

  We have:
$$
 Res_H  {\widetilde W} = 
   R_1 + \dots + R_{ r' }   + \widehat C_{s-s'+1} + \dots +\widehat C_{ s }  ;
$$
$$
 Res_H  {\widetilde T} =  R_1 + \dots + R_{ r' }   + \widehat C_{s-s'+1} + \dots +\widehat C_{ s } ;
$$

for $t$ even:
$$
 Tr_H  {\widetilde W} = 
 C_1 + \dots + C_{ r' }   + \widehat C_{r'+1} + \dots +\widehat C_{ s-s' } +P_1+\dots P_r +  $$$$
 + Tr_H ( \widehat C_{s-s'+1} + \dots +\widehat C_{ s } ) ;
$$
$$
 Tr_H  {\widetilde T} =  C_1 + \dots + C_{ r' }   + \widehat C_{r'+1} + \dots +\widehat C_{ s-s' } +M +$$
 $$
 + Tr_H ( \widehat C_{s-s'+1} + \dots +\widehat C_{ s } );
$$

for $t$ odd:
$$
 Tr_H  {\widetilde W} = 
 C_1 + \dots + C_{ r' }   + \widehat C_{r'+1} + \dots +\widehat C_{ s-s' } +P_1+\dots P_r +  $$$$
 + Tr_H ( \widehat C_{s-s'+1} + \dots +\widehat C_{ s } ) +  M;
$$
$$
 Tr_H  {\widetilde T} =  C_1 + \dots + C_{ r' }   + \widehat C_{r'+1} + \dots +\widehat C_{ s+1-s' } + $$
 $$
 + Tr_H ( \widehat C_{s-s'+1} + \dots +\widehat C_{ s } ).
$$

Hence $Res_H  {\widetilde W} $ and $Res_H  {\widetilde T} $ are the union of 
$s'$ generic sundials and the $r'$ points $ R_1, \dots , R_{ r' } $.  As in Case a), these $r'$ points are generic points lying  on $H$, and for $d \leq 5$ these points are generic also on $\mathbb P^n$. So  for $d \leq 5$ by the inductive hypothesis we get 
$$\dim (I_{Res_H  {\widetilde W} })_{d-1} =\dim (I_{Res_H  {\widetilde T} })_{d-1}=0.
$$
For $d>5$, consider $Res_H  {\widetilde W } -(R_1 + \dots + R_{ r' }) $ and 
$Res_H  {\widetilde T } -(R_1 + \dots + R_{ r' })$.
These schemes are the union of 
$s'$ generic sundials, hence by the inductive hypothesis we have
$$\dim (I_{Res_H  {\widetilde W} -(R_1 + \dots + R_{ r' })})_{d-1} =\dim (I_{Res_H  {\widetilde T} -(R_1 + \dots + R_{ r' })})_{d-1}=r'.
$$
Moreover  (see the  Appendix, Lemma \ref{disuguaglianza1} for computation):
$$\dim (I_{Res_H  {\widetilde W} -(R_1 + \dots + R_{ r' })+H })_{d-1} =
\dim (I_{Res_H  {\widetilde T} -(R_1 + \dots + R_{ r' })+H })_{d-1} 
$$
$$=\max \left \{
{d-2+n \choose n} - t'(d-1) ; 0 \right \}= 0.$$
Hence, by
 Lemma \ref{AggiungerePuntiSuY} (with $Y=H$ ) we get 
 $$\dim (I_{Res_H  {\widetilde W} })_{d-1} =\dim (I_{Res_H  {\widetilde T} })_{d-1}=0.
$$ 
 
 Now we compute $\dim (I_{Tr_H  {\widetilde W} })_{d} $.
 
  Recall that
 $$
 Tr_H  {\widetilde W} = 
 C_1 + \dots + C_{ r' }   + \widehat C_{r'+1} + \dots +\widehat C_{ s-s' } +P_1+\dots P_r +  $$$$
 + Tr_H ( \widehat C_{s-s'+1} + \dots +\widehat C_{ s } ) + \hbox{ (if $t $  is odd) }  M;
$$
 so the trace $Tr_H  {\widetilde W} $ is the  union of  $r'$ degenerate conics, $s-s'-r'$ sundials,
 $r+2s'=r+t'$ generic points and, if $t$ is odd, a generic line. Since $r+t' \geq r'$ (see the Appendix, Lemma \ref{ss'rr'}(b)), by Remark \ref{degenerare2rette}, the dimension in degree $d$ of $Tr_H  {\widetilde W} $ is not more than the dimension in   degree $d$ of a scheme which is the union of $s-s'$ sundials and 
 $r+t'-r'$ generic points and, if $t$ is odd, a generic line. That is, by the inductive hypothesis,
  $$\dim (I_{Tr_H  {\widetilde W} })_{d}$$
  $$
   \leq {d+n-1 \choose {n-1}}- 2(s-s')(d+1)-(r+t'-r')+ \hbox{(if $t$ is odd) } (d+1)
  $$
  $$={d+n-1 \choose {n-1}}- (t-t')(d+1)-  {d+n \choose {n}}+t(d+1)-t'+
  $$
  $$+{d+n-1 \choose {n}}-t'd =0.
  $$
  It follows that $\dim (I_{Tr_H  {\widetilde W} })_{d} =0$.
  
 Now  we compute  $\dim (I_{Tr_H  {\widetilde T} })_{d} $. 
Recall that for $t$ odd we have
$$
 Tr_H  {\widetilde T} =  C_1 + \dots + C_{ r' }   + \widehat C_{r'+1} + \dots +\widehat C_{ s+1-s' } + $$
 $$
 + Tr_H ( \widehat C_{s-s'+1} + \dots +\widehat C_{ s } );
$$
hence in this case the trace $Tr_H  {\widetilde T} $ is the union of  $r'$ degenerate conics, $s+1-s'-r'$ sundials and  $2s'$ generic points. Since $2s' \geq r'$ (see the Appendix, Lemma \ref{ss'rr'}(c)), by Remark \ref{degenerare2rette}, the dimension in degree $d$ of $Tr_H  {\widetilde T} $ is not more than the dimension in   degree $d$ of a scheme which is the union of $s+1-s'$ sundials  and 
 $2s'-r'$ generic points. Hence, by the inductive hypothesis,
 $$\dim (I_{Tr_H  {\widetilde T} })_{d} \leq {d+n-1 \choose {n-1}}- 2(s+1-s')(d+1) - (2s'-r')
  $$
   $$= {d+n-1 \choose {n-1}}- (t-t' +1)(d+1) - t'+r'
  $$
     $$= {d+n-1 \choose {n-1}}- t(d+1)+t'd-(d+1) +{d+n-1 \choose {n}}-t'd
  $$
  $$={d+n \choose {n}}- t(d+1)-(d+1)  = r-(d+1) \leq 0.$$
\\

For $t$ even we have
$$
 Tr_H  {\widetilde T} =  C_1 + \dots + C_{ r' }   + \widehat C_{r'+1} + \dots +\widehat C_{ s-s' } +M +$$
 $$
 + Tr_H ( \widehat C_{s-s'+1} + \dots +\widehat C_{ s } ).
$$
hence the trace $Tr_H  {\widetilde T} $ is the union of  $r'$ degenerate conics, $s-s'-r'$ sundials, a generic line and 
 $2s'$ generic points. Since $2s' \geq r'$ (see the Appendix, Lemma \ref{ss'rr'}(c)), by Remark \ref{degenerare2rette}, the dimension in degree $d$ of $Tr_H  {\widetilde T} $ is not more than the dimension in   degree $d$ of a scheme which is the union of $s-s'$ sundials, a generic line  and 
 $2s'-r'$ generic points. So, by the inductive hypothesis,
  $$\dim (I_{Tr_H  {\widetilde T} })_{d} \leq {d+n-1 \choose {n-1}}- 2(s-s')(d+1)-(d+1) - (2s'-r')
  $$
   $$= {d+n-1 \choose {n-1}}- (t-t')(d+1)-(d+1) - t'+r'
  $$
     $$= {d+n-1 \choose {n-1}}- t(d+1)+t'd-(d+1) +{d+n-1 \choose {n}}-t'd
  $$
  $$={d+n \choose {n}}- t(d+1)-(d+1)  = r-(d+1) \leq 0.$$

  It follows that $\dim (I_{Tr_H  {\widetilde T} })_{d} =0$.

So we have proved that 
  $$\dim (I_{Res_H  {\widetilde W} })_{d-1} = \dim  (I_{Tr_H  {\widetilde W} })_{d}=0,
  $$
  $$\dim  (I_{Res_H  {\widetilde T} })_{d-1}=
 \dim   (I_{Tr_H  {\widetilde T} })_{d}=0,$$
 then by Lemma \ref{Castelnuovo} we are done.
\\

{\it Case c):} $t$ even and $t'$ odd.\\
In this case we have

$$W =
  \widehat C_1 + \dots +\widehat  C_{  s }  +P_1+\dots +P_r  \ ,     $$
     $$T =
  \widehat C_1 + \dots +\widehat  C_{   s }  +M    .  $$

  Notice that  $s-s'-r' -1\geq 0$ (this inequality is treated in the Appendix,  Lemma \ref{ss'rr'}(a)).  As in cases {\it a)} and  {\it b)}, we can use this  inequality to construct  $\widetilde W$, a specialization of $W$, as follows: pick  $H$ a generic hyperplane, and
  \\
$\bullet$  specialize the  $r'$ sundials  $\widehat C_1, \dots , \widehat C_{r'}$, in such a way that $L_{i,1}+L_{i,2}\subset H$, but   $2R_i|_{H_i} \not\subset H $, for $1 \leq i \leq r'$; \\
 $\bullet$  specialize  the $s - s' - r' -1$  sundials   $\widehat C_{r'+1}$,  $\dots, \widehat C_{s-s' -1}$ into $H$; \\
   $\bullet$ specialize the  sundial $\widehat C_{s-s'}$  in such a way that $L_{s-s',1} \subset H$,  but  $L_{s-s',2} \not\subset H $;
 \\
  $\bullet$ specialize the points $P_1, \dots, P_r$ into $H$.
  
Let  $\widetilde T$ be the scheme obtained from $T$  by  specializing: \\
 $\bullet$  the first  $r'$ sundials,  $\widehat C_1, \dots , \widehat C_{r'}$, in such a way that $L_{i,1}+L_{i,2}\subset H$, but  $2R_i|_{H_i} \not\subset H $; \\
 $\bullet$   the next $s - s' - r'$  sundials ,  $\widehat C_{r'+1}$,  $\dots, \widehat C_{s-s'}$,
 into  $H $.\\
  
  As in  cases {\it a)} and  {\it b)}, we will prove that 
  $$\dim (I_{Res_H  {\widetilde W} })_{d-1} = \dim  (I_{Tr_H  {\widetilde W} })_{d}=\dim  (I_{Res_H  {\widetilde T} })_{d-1}=
 \dim   (I_{Tr_H  {\widetilde T} })_{d}=0.$$

  We have:
$$
 Res_H  {\widetilde W} = 
   R_1 + \dots + R_{ r' }   + \widehat C_{s-s'+1} + \dots +\widehat C_{ s } +L_{s-s',2} ;
$$
$$
 Res_H  {\widetilde T} =  R_1 + \dots + R_{ r' }   + \widehat C_{s-s'+1} + \dots +\widehat C_{ s } +M;
$$
$$
 Tr_H  {\widetilde W} = 
 C_1 + \dots + C_{ r' }   + \widehat C_{r'+1} + \dots +\widehat C_{ s-s' -1} +P_1+\dots P_r +  $$$$
 +L_{s-s',1}   +Tr_H ( 2R_{s-s'}|_{H_{s-s'}} +\widehat C_{s-s'+1} + \dots +\widehat C_{ s } );
$$
$$
 Tr_H  {\widetilde T} =  C_1 + \dots + C_{ r' }   + \widehat C_{r'+1} + \dots +\widehat C_{ s-s' } + $$
 $$
 + Tr_H ( \widehat C_{s-s'+1} + \dots +\widehat C_{ s } +M).
$$
So we have that $Res_H  {\widetilde W} $ and $Res_H  {\widetilde T} $ are the  union of 
$s'$ generic sundials, a generic  line and the $r'$ points $ R_1, \dots , R_{ r' } $.  These $r'$ points 
lie on $H$, and are  generic points  on $H \simeq \mathbb P^{n-1}$, so as  in cases {\it a)} and  {\it b)},  for $d \leq 5$ by the inductive hypothesis we immediately get 
$$\dim (I_{Res_H  {\widetilde W} })_{d-1} =\dim (I_{Res_H  {\widetilde T} })_{d-1}=0.
$$
For $d>5$, consider $Res_H  {\widetilde W } -(R_1 + \dots + R_{ r' }) $ and 
$Res_H  {\widetilde T } -(R_1 + \dots + R_{ r' })$.
By the inductive hypothesis we have
$$\dim (I_{Res_H  {\widetilde W} -(R_1 + \dots + R_{ r' })})_{d-1} =\dim (I_{Res_H  {\widetilde T} -(R_1 + \dots + R_{ r' })})_{d-1}=r'.
$$
Moreover 
$$\dim (I_{Res_H  {\widetilde W} -(R_1 + \dots + R_{ r' })+H })_{d-1} =
\dim (I_{Res_H  {\widetilde T} -(R_1 + \dots + R_{ r' })+H })_{d-1} 
$$
$$=\max \left \{
{d-2+n \choose n} - (2s'+1)(d-1) ; 0 \right \}=$$
$$=\max \left \{
{d-2+n \choose n} - t'(d-1) ; 0 \right \}= 0,$$
 (the last equality is proved in the Appendix, Lemma \ref{disuguaglianza1} ). 
 
 Hence, by
 Lemma \ref{AggiungerePuntiSuY} (with $Y=H$ )  we get 
 $$\dim (I_{Res_H  {\widetilde W} })_{d-1} =\dim (I_{Res_H  {\widetilde T} })_{d-1}=0.
$$
 
 Now we compute $\dim (I_{Tr_H  {\widetilde W} })_{d} $ and $\dim (I_{Tr_H  {\widetilde T} })_{d} $.
 
 Recall that 
 $$
 Tr_H  {\widetilde T} =  C_1 + \dots + C_{ r' }   + \widehat C_{r'+1} + \dots +\widehat C_{ s-s' } + $$
 $$
 + Tr_H ( \widehat C_{s-s'+1} + \dots +\widehat C_{ s } +M).
$$
Hence the trace $Tr_H  {\widetilde T} $ is the  union of  $r'$ degenerate conics, $s-s'-r'$ sundials,   and 
 $2s'+1$ generic points.  Since $2s'+1 \geq r'$ (see the Appendix, Lemma \ref{ss'rr'}(c)), by Remark \ref{degenerare2rette}, the dimension, in degree $d$, of $Tr_H  {\widetilde T} $ is not more than the dimension, in   degree $d$, of a scheme which is the union of $s-s'$ sundials and
 $2s'-r'$ generic points. Thus, by the inductive hypothesis,
  $$\dim (I_{Tr_H  {\widetilde T} })_{d} \leq {d+n-1 \choose {n-1}}- 2(s-s')(d+1)- (2s'+1-r')
  $$
     $$= {d+n-1 \choose {n-1}}- t(d+1)+t'd-d -1+r' 
  $$
    $$= {d+n-1 \choose {n-1}}- t(d+1)+t'd-d -1+{d+n-1 \choose {n}}-t'd
  $$
  $$={d+n \choose {n}}- t(d+1)-d  -1= r-d-1 \leq 0.$$

  It follows that $\dim (I_{Tr_H  {\widetilde T} })_{d} =0$.

Finally, recall that 
$$
 Tr_H  {\widetilde W} = 
 C_1 + \dots + C_{ r' }   + \widehat C_{r'+1} + \dots +\widehat C_{ s-s' -1} +P_1+\dots P_r +  $$
 $$
 +L_{s-s',1}  +Tr_H ( 2R_{s-s'}|_{H_{s-s'}}  +\widehat C_{s-s'+1} + \dots +\widehat C_{ s } );
$$
 so the trace $Tr_H  {\widetilde W} $ is the  union of  $r'$ degenerate conics, $s-s'-r'-1$ sundials,
 $r+2s'=r+t'-1$ generic points., and a line with an embedded point (that is the scheme 
 $L_{s-s',1}  +2R_{s-s'}|_{H_{s-s'}\cap H} $).
 
 Let $\bar L \subset H$ be a generic  line  through $R_{s-s'}$ and let  $\widehat C$ denote  the sundial 
 $\bar L+L_{s-s',1}  +2R_{s-s'}|_{H_{s-s'}\cap H} $.
 
Now specialize $d$ of the $r+t'-1$ generic points of $Tr_H  {\widetilde W}$ onto the  line $\bar L $, so that the hypersurfaces defined by the forms of 
$(I_{Tr_H  {\widetilde W}})_d$  have
the sundial $\widehat C$ in their base locus.
 
 Since $r+t' -1-d \geq r'$ (see the Appendix, Lemma \ref{ss'rr'}(b)), by Remark \ref{degenerare2rette}, the dimension, in degree $d$, of $Tr_H  {\widetilde W} $ is not more than the dimension, in degree $d$, of a scheme which is the  union of  $r+t'-1-d-r'$ generic points and the $s-s'$ sundials
 $$ 
 \widehat C_1 + \dots + \widehat C_{ r' }   + \widehat C_{r'+1} + \dots +\widehat C_{ s-s' -1}
 + \widehat C. $$
 Thus, by the inductive hypothesis,
  $$\dim (I_{Tr_H  {\widetilde W} })_{d} \leq {d+n-1 \choose {n-1}}- 2(s-s')(d+1)-(r+t'-1-d-r')
  $$
  $$={d+n-1 \choose {n-1}}- (t-t'+1)(d+1)-  {d+n \choose {n}}+t(d+1)-t'+1+d
  $$
  $$+{d+n-1 \choose {n}}-t'd =0.
  $$
  It follows that $\dim (I_{Tr_H  {\widetilde W} })_{d} =0$.
  So we have proved that 
  $$\dim (I_{Res_H  {\widetilde W} })_{d-1} = \dim  (I_{Tr_H  {\widetilde W} })_{d}=0,
  $$
  $$\dim  (I_{Res_H  {\widetilde T} })_{d-1}=
 \dim   (I_{Tr_H  {\widetilde T} })_{d}=0.$$
By Lemma \ref{Castelnuovo} we are done.

\end{proof}

\medskip
%%%%%%%%%%%%%%%%%%%%%%%%%%%
%%%%%%%%%%%%%%%%%%%%%%%%%%%%%
%%%%%%%%%%%%%%%%%%%%%%%%%%%%%%%

\section{Appendix}

\begin{lem}  \label{ss'rr'}  Let $n \geq 4$, $d \geq 2$,
$$t=  \left\lfloor{ {d+n \choose n} \over {d+1} } \right\rfloor, 
  \ \ \ \ \   r= {d+n \choose n} - t(d+1)
\ \ \ \ \     s = \left\lfloor{ \frac {t} 2 } \right\rfloor .$$

$$t'=  \left\lfloor{ {d-1+n \choose n} \over {d} } \right\rfloor, 
  \ \ \ \ \   r'= {d-1+n \choose n} - t'd
\ \ \ \ \     s' = \left\lfloor{ \frac {t'} 2 } \right\rfloor .$$
then 

\begin{itemize}
\item[(a)]
\begin{itemize} \item[$\bullet$]  for $t$ odd and  $t'$ odd, or for $t'$ even: $s-s'-r' \geq0;$
\item[$\bullet$]   for $t$ even and  $t'$ odd: $s-s'-r' -1\geq 0$
\end{itemize}
\medskip

\item[(b)] \begin{itemize} \item[$\bullet$]   for $t$ odd and  $t'$ odd, or for $t'$ even:
 $r+t' \geq r' ;$ 
 \item[$\bullet$] for $t$ even and  $t'$ odd: 
 $r+t' -1-d \geq r'$ 
 \end{itemize}
\medskip

\item [(c)] \begin{itemize} \item[$\bullet$] 
$2s' \geq r' .$
  \end{itemize}

\end{itemize}

\end{lem}

\begin{proof} (a)  For $t$ odd and  $t'$ odd, or for $t'$ even,
since $s-s'-r' =  \left\lfloor  {\frac t 2} \right\rfloor -  \left\lfloor  {\frac {t'} 2} \right\rfloor -r'$, 
 it sufficies to verify that 
 $$t-1-t'-2r' \geq 0.$$ 
 In case $t$ even and  $t'$ odd, we have $s-s'-r' -1=  {\frac t 2}-  {\frac {t'-1} 2} -r'-1 =
  {\frac 12} (t-t'-2r'-1) $, hence also in this case  it is enough to verify that 
 $t-1-t'-2r' \geq 0.$ 
 
 We have
 
 $$t-t'-1-2r'  ={ \frac {{d+n \choose n} -r} {d+1} } - { \frac {{d+n-1 \choose n} -r'} {d} } -1-2r' \geq 0 
$$
$$\iff  d {d+n \choose n} -rd -  (d+1) {d+n-1 \choose n} +r'(d+1) -(1+2r')d(d+1) \geq 0 
$$
$$\iff  {d+n-1 \choose n}(n-1) -rd - r' (d+1)(2d-1)-d(d+1) \geq 0 
$$
If $n \geq 5$, since $r \leq d$, $r' \leq d-1$ and $d \geq 2$ we have 
$$  {d+n-1 \choose n}(n-1) -rd - r' (d+1)(2d-1)-d(d+1) 
$$
$$ \geq 4 {d+4 \choose 5} -d^2 - (d-1)(d+1)(2d-1)-d(d+1) 
$$
$$ = {\frac 1 {30}} (d+4)(d+3)(d+2)(d+1)d - 2d^2(d+1)+d^2+d-1$$
$$
 ={\frac 1 {30}} d(d+1) (d-1)(d-2)(d+12) +d^2+d-1\geq 5,$$
 and we are done for $n \geq5$.
 
 For $n=4$, since $r \leq d$ and  $r' \leq d-1$ we have
 $$  {d+n-1 \choose n}(n-1) -rd - r' (d+1)(2d-1)-d(d+1) 
$$
$$ \geq 3 {d+3 \choose 4} -d^2 - (d-1)(d+1)(2d-1)-d(d+1) 
$$
$$ = {\frac 1 {8}}d(d+1)((d-1)(d-10)+4)-1,$$
so for $d \geq 10$ we are done. 

For $n=4$ and $2 \leq d \leq 9$ we have:
\\

$\begin{matrix}
d      &   \ t       &  \  t'      &  \  r'  &  \    s - s'-r'  &  \   s - s'-r'  -1 \\
\\
2      &  \  5       &  \  2 & \   1   &     \            0         &     \              \\
3      & \   8       &  \  5 & \   0   &     \     &     \    1  \\
4      & \   14        &  \  8 & \   3   &      \    0  \\
5      & \   21        & \   14   &    \  0  &  \    3  \\
6      & \   30  &  \  21 & \   0  &      \    &      \    4  \\
7      & \   41  &  \  30 & \   0  &    \  5    \\
8      & \    55  &  \  41 &     \  2  &  \    5  \\
9      & \   71  &  \  55 & \   0    &  \    8 \\
\end{matrix}
$
\\

This table shows that for $t$  even and  $t'$  odd,
we have $s-s'-r' -1\geq 0$, while  for $t$ odd and  $t'$ odd, or for $t'$ even we have
 $s-s'-r' \geq0, $ and this completes the proof of  (a).
\\

(b) For $t$ odd and  $t'$ odd, or for $t'$ even we have to prove that
$r+t'\geq r' .$
We have
$$r+t'\geq r'  \iff  r+{ \frac  {{d+n-1 \choose n}-r' } {d} }  -r' \geq 0$$
$$ \iff  rd+  {d+n-1 \choose n}-r'(d+1)\geq 0.$$
Since $n \geq4$, $r \geq 0$ , $r' \leq d-1$ and $d \geq 2$ we have
$$ rd+  {d+n-1 \choose n}-r'(d+1) \geq  {d+3 \choose 4}-(d-1)(d+1)$$
$$=\frac 1 {24} (d+1)((d-2)(d^2+7d-4)+16)\geq 2,$$
so it follows that  $r+t'\geq r' .$
\\

For $t$  even and  $t'$  odd we have to show that $r+t' -1-d \geq r'$, or, equivalently,
$$  rd+  {d+n-1 \choose n}-(r'+1)(d+1)\geq 0.$$
For $n \geq5$, since $r \geq 0$ , $r' \leq d-1$ and $d \geq 2$ we have
$$  rd+  {d+n-1 \choose n}-(r'+1)(d+1) $$
$$\geq     {d+4 \choose 5}-d(d+1) = \frac 1 {120} d(d+1)((d+4)(d+3)(d+2)-120)\geq 0.$$
Finally let $n=4$. If $d=2$, then we are not in the case $t$  even and  $t'$  odd, so  we may assume that $d\geq 3$.

Since  $d\geq 3$,  $r \geq 0$ , $r' \leq d-1$ we get
$$  rd+  {d+n-1 \choose n}-(r'+1)(d+1) $$
$$\geq     {d+3 \choose 4}-d(d+1) = \frac 1 {24}d(d+1)( (d+3)(d+2)-24)\geq 3,$$
and (b) is proved.
\\

(c) We have 
 $$2s' - r' \geq 0 \iff t'-1-r' \geq 0 \iff   {d+n-1 \choose n} -r' (d+1)-d \geq 0.$$
Since  $r' \leq d-1$ we  get 
$$ {d+n-1 \choose n} -r' (d+1)-d \geq
 {d+n-1 \choose n}  -d(d+1)+1
$$
which, for $n \geq 5$ and $d \geq 2$, or for $n \geq 4$ and $d \geq 3$, is positive (see the case (b) above for computation).
So we are left with the case $n=4$, $d=2$. In this case we have
 $2s' - r'  = 1,$
and we are done.

\end{proof}

\begin{lem}  \label{disuguaglianza1}  Let $n \geq 4$, $d > 5$, 
$$t'=  \left\lfloor{ {d-1+n \choose n} \over {d} } \right\rfloor, 
  \ \ \ \ \   r'= {d-1+n \choose n} - t'd
\ \ \ \ \     s' = \left\lfloor{ \frac {t'} 2 } \right\rfloor .$$
Then
$$
\max \left \{
{d-2+n \choose n} - t' (d-1) ; 0 \right \}=0.$$
\end{lem}

\begin{proof} It is enough to verify that 
${d-2+n \choose n} - t'(d-1) \leq 0.$

We have 
$$
 {d-2+n \choose n} - t'(d-1) =
{d-2+n \choose n} - \frac {({d-1+n \choose n}-r')} d (d-1)
$$
$$
= {\frac  1 d }\left (
d{d-2+n \choose n} - (d-1) {d-1+n \choose n}+ r'  (d-1)
\right )
$$
$$
= {\frac  1 d }\left (
{d-2+n \choose n} (-1+n)+ r'  (d-1)
\right ).
$$

Since 
$$
{d-2+n \choose n} (n-1)- r'  (d-1) \geq 
3{d-2+4 \choose 4} -  (d-1)^2 $$
$$ \geq {\frac {d-1} 8} (d(d+1)(d+2)-8(d-1)) \geq 0
,
$$
we get the conclusion.

\end{proof}

%%%%%%%%%%%%%%%%%%%%%%%%%%%%
%%%%%%%%%%%%%%%%%%%%%%%%%%
%%%%%%%%%%%%%%%%%%%%%

\end{document}